\documentclass{amsart}
\usepackage{amsthm}
\usepackage{amsmath,mathtools, amssymb, amsfonts,mathrsfs,graphicx,amsmath,bbold}
\usepackage[usenames,dvipsnames]{xcolor}
\usepackage[mathscr]{eucal}
\usepackage[backgroundcolor=yellow,shadow,colorinlistoftodos]{todonotes}
\usepackage[usenames,dvipsnames]{pstricks}
\usepackage[pagewise]{lineno}
\usepackage[shortlabels]{enumitem}
\usepackage{array,multicol,multirow,tabularx}
\usepackage{tikz}
\usepackage{tikz-cd}
\usepgflibrary{arrows}
\usetikzlibrary{matrix}
\usetikzlibrary{decorations.pathreplacing,decorations.markings}
\usepgflibrary{shapes.geometric}
\usetikzlibrary{shapes.geometric}
\usetikzlibrary{positioning}
\usepackage[all]{xy}

\makeatletter
\providecommand\@dotsep{5}
\def\listtodoname{List of Todos}
\def\listoftodos{\@starttoc{tdo}\listtodoname}
\makeatother

\theoremstyle{plain}
\newtheorem{theorem}{Theorem}
\newtheorem*{theorem*}{Theorem}
\newtheorem*{HTT}{Homotopy Transfer Theorem}
\newtheorem*{mainthm*}{Main Theorem}

\newtheorem{proposition}[theorem]{Proposition}
\newtheorem{lemma}[theorem]{Lemma}

\newtheorem*{lemma*}{Lemma}
\newtheorem{corollary}[theorem]{Corollary}

\newtheorem*{cor*}{Corollary}

\theoremstyle{definition}
\newtheorem{definition}[theorem]{Definition}

\newtheorem*{ass*}{Assumption}
\newtheorem*{conventions*}{Conventions}
\newtheorem*{acknowledgments*}{Acknowledgment}

\newtheorem*{notation*}{Notation}
\newtheorem{remark}[theorem]{Remark}

\newcounter{Q_num}
\theoremstyle{plain}
\newtheorem{question}[Q_num]{Question}
\newtheorem*{question*}{Question}

\def\Trans#1#2#3{{\sf Tr}_{#1,#2,#3}}
\def\setinf{{\mathcal A}_\infty}
\def\+{\!+\!}
\newcommand{\Pinf}{\cP_\infty}

\def\bfalpha{{\boldsymbol \alpha}}
\def\bfbeta{{\boldsymbol \beta}}
\def\pa{\partial}
\def\bfF{{\boldsymbol F}}

\def\bfh{{\boldsymbol h}}
\def\bfg{{\boldsymbol g}}
\def\qi{{quasi-isomorphism}}

\def\bfpsi{{\boldsymbol \psi}}

\def\bfomega{{\boldsymbol \omega}}
\def\bfphi{{\boldsymbol \phi}}

\def\bfT{{\boldsymbol \theta}}
\def\bff{{\boldsymbol f}}
\def\bfnu{{\boldsymbol \nu}}
\def\bfmu{{\boldsymbol \mu}}
\def\Trp{{\tt P}}
\def\otexp#1#2{#1^{\otimes #2}} 
\def\ot{\otimes}

\def\p{\bm{{p}}}

\def\krouzekstopka{
\unitlength .9mm
\linethickness{0.4pt}
\ifx\plotpoint\undefined\newsavebox{\plotpoint}\fi 
\begin{picture}(0,3)
\put(0,4.23){\line(0,-1){2}}
\put(0,0){\line(0,-1){2}}
\put(0,1){\makebox(0,0)[cc]{\Large $\circ$}}
\end{picture}
}

\newcommand{\Z}{\mathbb{Z}}
\newcommand{\kk}{\Bbbk}
\newcommand{\eR}{R}

\newcommand{\xto}[1]{\xrightarrow{#1}}

\newcommand{\el}{\ell}

\newcommand{\tha}{\theta}
\newcommand{\Tha}{\Theta}
\newcommand{\vphi}{\varphi}

\def\martin#1\endmartin{\noindent{{\bf Martin:}\ {\color{rgb:-green!4!yellow,30;green!40!yellow,2;red,62} #1} 
    \hfill\rule{10mm}{.75mm}} \break} 

\newcommand{\ti}[1]{\tilde{#1}}
\newcommand{\ba}[1]{\bar{#1}}

\newcommand{\sse}{\subseteq}

\newcommand{\cP}{\mathcal{P}}
\newcommand{\antish}{\text{\rotatebox[origin=c]{180}{$!$}}}
\newcommand{\coP}{\mathcal{P}^\antish}

\newcommand{\maps}{\colon}
\newcommand{\tensor}{\otimes}

\newcommand{\gen}[1]{\bigl \langle #1 \bigr\rangle}

\newcommand{\T}{\bar{T}}
\newcommand{\Tc}{\bar{T}^c}
\newcommand{\rDelta}{\bar{\Delta}}
\newcommand{\rdDelta}[1]{\bar{\Delta}_{(#1)}}

\newcommand{\del}{\partial}

\newcommand{\bs}{\mathbf{s}}

\newcommand{\ua}{\bs}
\newcommand{\da}{\bs^{-1}}

\newcommand{\cF}{\mathcal{F}}
\newcommand{\cC}{\mathcal{C}}

\newcommand{\dlt}{\delta}
\newcommand{\Om}{\Omega}
\newcommand{\Del}{\Delta}

\newcommand{\cc}{\circ}

\newcommand{\wti}[1]{\widetilde{#1}}

\newcommand{\simten}[2]{{#1}_1 \tensor {#1}_2 \tensor \cdots \tensor {#1}_{#2}}

\newcommand{\plim}{\varprojlim}
\newcommand{\Find}[2]{{#1_{(#2)}}^{\!  1}}
\newcommand{\J}{\mathfrak{I}}
\newcommand{\eps}{\varepsilon}

\newcommand{\jm}{\jmath}

\DeclareMathOperator{\Hom}{\mathrm{Hom}}

\DeclareMathOperator{\id}{\mathrm{id}}

\DeclareMathOperator{\dR}{\mathrm{dR}}

\DeclareMathOperator{\cQ}{\mathcal{Q}}
\DeclareMathOperator{\chark}{\mathrm{char}}
\usepackage{bm}
\newcommand{\cM}{\mathscr{H}}
\newcommand{\cR}{\mathcal{R}}
\newcommand{\btensor}{\bm{\otimes}}
\newcommand{\tlab}[1]{{\scriptsize #1}}

\usepackage[normalem]{ulem} 

\newcommand*{\emptycomment}[1]{}

\usepackage[breaklinks=true,colorlinks=true,citecolor=blue,urlcolor=blue,linkcolor=black]{hyperref}

\begin{document}
\title{Which homotopy algebras come from transfer?}

\author{Martin Markl}
\address{The Czech Academy of Sciences, Institute of Mathematics, {\v Z}itn{\'a} 25,
         115 67 Prague, The Czech Republic}
\email{markl@math.cas.cz}

\author{Christopher L.\  Rogers} 
\address{
 Department of Mathematics and Statistics, University of Nevada,
 Reno.\newline  1664 N. Virginia Street Reno, NV 89557-0084 USA} 
\email{chrisrogers@unr.edu, chris.rogers.math@gmail.com}

\thanks{This material is based upon work supported by the National
Science Foundation under Grant No.~DMS-1440140 while the authors were
in residence at the Mathematical Sciences Research Institute in
Berkeley, California, during the Spring 2020 semester. The first author was also supported by 
grant GA \v CR 18-07776S, Praemium Academi\ae\ and RVO: 67985840. The second author was also supported by a grant from the Simons Foundation/SFARI (585631,CR)}


\begin{abstract}
  We characterize $A_\infty$-structures that are equivalent to a given
  transferred structure over a chain homotopy equivalence or a
  quasi-isomor\-phism, answering a question posed by D.\ Sullivan. Along
  the way, we present an obstruction theory for weak
  \hbox{$A_\infty$-morphisms} over an arbitrary commutative ring.  We
  then generalize our results to $\cP_\infty$-structures over a field
  of characteristic zero, for any quadratic Koszul operad $\cP$.
\end{abstract}

\maketitle



\section{Introduction}
\label{sec:intro}
An $A_\infty$-
algebra is 
a~homotopical generalization of a differential graded associative
algebra \cite{stasheff:TAMS63}. 
It is a chain complex $(A,d)$ equipped with a binary operation $\mu_2$ for which
the associative law only holds up to a specified chain homotopy $\mu_3$. This
homotopy is taken to be part of the structure; it too must satisfy a
law, but only up to another specified homotopy $\mu_4$, which satisfies yet
another law and so forth. See \cite[Sec.\ 2]{Markl:2006}
for the precise definition and terminology. This seemingly complicated
generalization is in fact quite natural, and it endows 
$A_\infty$-algebras with many desirable homological properties.

For example, given a dg associative algebra $A$ and a 
chain homotopy equivalent complex $A'$, there is in general no
dg associative  algebra structure on $A'$ such that the given chain homotopy
equivalence becomes a morphism of dg associative algebras. On
the other hand, for \hbox{$A_\infty$-algebras} one has:

\begin{HTT}[{\cite[Theorem~10.3.1]{LV}}]
Let the chain complex $(A',d')$ be a homotopy retract of $(A,d)$,
i.e.\ there exists a diagram
\begin{align} 
\label{eq:htt}
&\xymatrix@1{     *{ \quad \ \  \quad (A, d)\ } 
\ar@(dl,ul)[]^{h}\ \ar@<0.5ex>[r]^(.6){f_1} & *{\
(A',d')}    \ar@<0.5ex>[l]^(.4){g_1}},\
g_1 f_1 - \id_{A}  =d  h+ h  d
\end{align}
in which $f_1$ and $g_1$ are chain maps, with $g_1$
inducing an isomorphism on homology, and $h$ is a chain homotopy
between $g_1f_1$ and the identity endomorphism of~$A$.\footnote{In
  other words, $g_1$ is a left homotopy inverse of $f_1$.} 
Then any $A_\infty$-algebra
structure on $(A, d)$ can be transferred to an  $A_\infty$-algebra
structure on $(A', d')$ such that $g_1$ extends to a weak $A_\infty$-morphism. 
\end{HTT}

The first author proved in \cite{Markl:2006} a much stronger result,
providing simple explicit formulas not only for the transferred
$A_\infty$-structure and an extension of $g_1$, but
also for extensions of $f_1$ and  $h$. Furthermore, 
$g_1$ is not required in \cite{Markl:2006} to induce an isomorphism on homology.  
The extensions of both $f_1$ and $g_1$ will play a crucial role in this work.

The transfers of $A_\infty$-structures over a
chain map admitting a left homotopy inverse, as given by the formulas
presented in \cite{kontsevich-soibelman:00,Markl:2006,merkulov:98} and recalled in Sec.\ \ref{sec:trans} below, 
have found applications in many contexts. For example, they have been used in geometry~\cite{aka,bara,budur,civ,poli},
 homological algebra~\cite{amo,ballard} and mathematical \hbox{physics~\cite{hi,zero}}.
It is therefore natural to ask which $A_\infty$-structures are
equivalent to a given transferred one. This was the question posed to the first author by
Dennis Sullivan during his visit to the Simons Center in June 2019. 
The aim of this note is to give an answer for the case when the chain
map $f_1$ in \eqref{eq:htt}
over which the transfer is performed is a {\it chain homotopy equivalence}
as in~\eqref{eq:hmtpy-data}. That is, in addition to the hypothesis that $g_1f_1$ is chain homotopic to the identity $\mathrm{id}_A$, we also assume that $f_1g_1$  is chain homotopic to the identity~$\mathrm{id}_{A'}$.
If the ground ring is a field, then this is the same as being a quasi-isomorphism\footnote{I.e., a chain map
  inducing a   homology isomorphism.}.
 

\begin{conventions*}
All algebraic objects in Sections \ref{sec:quest} - \ref{sec:obst}
are defined over a fixed commutative unital ring $R$, except Sec.\ \ref{sec:hmpty-cases} where $R$ is a field. In Section \ref{sec:Pinf}, we restrict to the case $R=\kk$, where $\kk$ is a field of characteristic zero.
All graded objects are $\Z$-graded and unbounded; we use homological conventions for all dg objects. Given graded $R$-modules $V$ and $W$, we denote by $\Hom_R(V,W)$ the graded $R$-module $\Hom_R(V,W)_n:= \prod_{k \in \Z} \Hom_{R\mathrm{Mod}}(V_k,W_{k+n})$, where $\Hom_{R\mathrm{Mod}}(-,-)$ denotes the internal hom in the category of $R$-modules. We denote by $\ua  V$ and $\da V$, the suspension and desuspension, respectively, of the graded module $V$. Concretely, $(\ua  V)_n:= V_{n-1}$ and
$(\da V)_n:=V_{n+1}$.
 
Conventions and notations for $A_\infty$-algebras and their weak and strict morphisms are taken from~\cite[Sec.\ 2]{Markl:2006}. In Sec.\ \ref{sec:Pinf}, which is separate from the rest of the paper, we assume some familiarity with Koszul operads and homotopy operadic algebras as in \cite[Ch.\ 10]{LV}. 
\end{conventions*}

\begin{acknowledgments*}
We express our gratitude to Jim Stasheff,
Dennis Sullivan and the referee for useful suggestions and comments that 
led to substantial improvement of our~paper.
\end{acknowledgments*}

\section{Summary of results} \label{sec:quest}

Suppose that $(A,d,\bfmu) = (A,d,\mu_2,\mu_3,\ldots)$
is an $A_\infty$-algebra, $(A',d')$ a chain complex, and $f_1 :  
(A,d) \to (A',d')$ a chain map which is
a chain homotopy equivalence.  
Then it is well known (see Sec.\ \ref{sec:trans})
that there exists a transferred $A_\infty$-structure
\begin{equation*}
(A',d',\bfnu) = (A',d',\nu_2,\nu_3,\ldots)
\end{equation*} 
on $(A',d')$ and a lift of $f_1$ to a weak $A_\infty$-morphism
$\bff = (f_1,f_2,f_3,\ldots) \maps (A,d,\bfmu) \to (A',d',\bfnu)$.
Now suppose  that 
\begin{equation*}
(A',d',\bfmu') = (A',d',\mu'_2,\mu'_3,\ldots)
\end{equation*} 
is another $A_\infty$-structure on $(A',d')$. 
For the purposes of exposition, let us begin with an approximation to
Sullivan's question.
\begin{question}
\label{Q1}
In the situation above, is the $A_\infty$-structure 
$\bfmu' =\{\mu'_2,\mu'_3,\ldots \}$ on the complex $(A',d')$
equivalent to a transferred structure? 
\end{question}


We need to specify what precisely 
``is equivalent'' and the adjective ``transferred'' mean in the above sentence. Let
us start with the former one; transferred structures
will be treated in the next section.
First, by ``is equivalent'', we could mean that
there exists a weak $A_\infty$-morphism   
\begin{equation}
\label{eq:is}
\bfphi = (\phi_1,\phi_2,\phi_3,\ldots) \maps (A',d',\bfnu) \to (A',d',\bfmu')
\end{equation}
such that one of the following cases is satisfied:
\renewcommand{\arraystretch}{1.25}
\setlength{\tabcolsep}{5pt}
\begin{center}
\begin{tabular}{|c||c|c| c| } 
 \hline
\multirow{2}{*}{{\bf Case}} &  Relationship between  & Criterion for chain map  & Criteria for  \\
& $(A',d',\bfnu)$ and $(A',d',\bfmu')$ & $\phi_1 \maps (A',d') \to (A',d')$ & higher maps $\phi_{k \geq 2}$   \\
\hline \hline \multicolumn{4}{ |c| }{{\it Strict Cases}} \\
 \hline  \bf 1 & equality & $\phi_1=\id_{A'}$ & $\phi_k=0$ $\forall k \geq 2$ \\ 
\hline \bf 2 & strictly isomorphic & $\phi_1$ is an  automorphism & $\phi_k=0$ $\forall k \geq 2$ \\  
\hline \hline \multicolumn{4}{ |c| }{{\it Weak Cases}} \\
\hline  \bf 3 & isotopic & $\phi_1=\id_{A'}$ & none \\ 
\hline \bf 4& weakly isomorphic & $\phi_1$ is an automorphism & none \\ 
\hline
\end{tabular}
\end{center}

Other relationships are possible. Recall that a weak
$A_\infty$-morphism such as $\bfphi$ in \eqref{eq:is} above is an {\bf
  $A_\infty$-quasi-isomorphism}, or {\bf $A_\infty$-quism} for short,
if $\phi_{1}$ is a quasi-isomorphism of chain complexes. We say
$(A',d', \bfnu)$ and $(A',d',\bfmu')$ are {\bf weakly equivalent} or,
depending on the context,
that they have the same {\bf homotopy type} if they are connected by a zig-zag of $A_\infty$-quisms. Then ``is'' in Question \ref{Q1} could also mean:
\begin{center}
\begin{tabular}{|c||c|c|} 
 \hline \multicolumn{3}{ |c| }{{\it Homotopical Cases}} \\
 \hline 
\multirow{2}{*}{{\bf Case}} &  Relationship between  & Criteria for weak morphisms   \\
& $(A',d',\bfnu)$ and $(A',d',\bfmu')$ &  between $(A',d',\bfnu)$ and $(A',d',\bfmu')$  \\
 \hline \hline  \bf 5a  & \multirow{2}{*}{$A_\infty$-quasi-isomorphic} & $\exists $ $A_\infty$-quism $ \bfphi \maps (A',d',\bfnu) \to (A',d',\bfmu')$ \\
\cline{1-1} \cline{3-3} \bf 5b &  &  $\exists $ $A_\infty$-quism $ \bfpsi \maps (A',d',\bfmu') \to (A',d',\bfnu)$ \\
\hline \bf 6& weakly equivalent & $\exists$ $A_\infty$-quisms $(A',d',\bfnu) \leftarrow \bullet \to (A',d',\bfmu')$  \\
 \hline
\end{tabular}
\end{center}

\subsection{Main results}\label{sec:results} 
There are seven variations of Question \ref{Q1} to consider.
Cases 1 and 2 involve comparing the isomorphism class of $(A',d',\bfmu')$ 
to that of $(A',d',\bfnu)$ in the category of $A_\infty$-algebras and strict morphisms. 
As we show in Sec.\ \ref{sec:strict-iso}, these turn out to be the only cases in which explicit formulas for the transferred structure actually matter. 

Cases 3 and 4 involve comparing isomorphism classes in the category of
$A_\infty$-algebras and weak morphisms, while the remaining three
concern isomorphism classes in the corresponding ``homotopy
category''. The characterization via isotopy, Case 3, is perhaps the
most interesting. In Thm.\ \ref{thm:isotopy}, we exhibit a precise
relationship between the isotopy class of a transfer and the
homotopy type of $(A,d,\bfmu)$.
In particular, any
$A_\infty$-structure which is a target of an $A_\infty$-quism is
isotopic to a~transferred one (Cor.\ \ref{cor:isotopy}).  Our proofs
of these results are based on an obstruction theory for
$A_\infty$-morphisms, which we develop in Sec.\ \ref{sec:obst}.

When $R$ is a field, we prove in Sec.\ \ref{sec:hmpty-cases} that  
the three homotopical cases, 5a, 5b, and 6, are all equivalent to the existence of an $A_\infty$-quism between the original $A_\infty$-structure $(A,d,\bfmu)$ and $(A',d',\bfmu')$. 

Furthermore, as we show in Cor.\ \ref{cor:iso-char}, our results for Case 3 also provide a positive answer to Sullivan's original question, which we can 
now state precisely:

\begin{question}
\label{Q2}
Can one formulate, in terms of the initial data $\bfmu$ and $f_1$ as above, 
the necessary and sufficient conditions for the $A_\infty$-structure 
$\bfmu'$ to be isotopic to a~structure transferred over the chain
homotopy equivalence  $f_1$?
\end{question}

Finally, in Sec.\ \ref{sec:Pinf}, we generalize Thm.\ \ref{thm:isotopy}
to the transfer of $\Pinf$-structures over a~field $\kk$ with $\chark \kk =0$, for any quadratic Koszul operad $\cP$.

\section{A reminder on transfers} \label{sec:trans}
We recall some basic features of transferred $A_\infty$-structures.
The initial data 
are an $A_\infty$-algebra $(A,d,\bfmu)=(A,d,\mu_2,\mu_3,\ldots)$ 
a chain complex  $(A',d')$, and
a chain map $f_1 \maps  (A,d) \to (A',d')$. A {\bf transfer} of
$(A,d,\bfmu)$ {\bf over a chain map} $f_1$
is an $A_\infty$-structure $(A',d',\bfnu) =
(A',d',\nu_2,\nu_3,\ldots)$ on 
$A'$ and an extension
\begin{equation}
\label{eq:f}
\bff = (f_1,f_2,f_3,\ldots) : (A,d,\mu_2,\mu_3,\ldots) \longrightarrow
(A',d',\nu_2,\nu_3,\ldots) 
\end{equation}
of the chain map $f_1$ into a weak $A_\infty$-morphism.


There are two standard situations in which such transfers are known to exist: 
the ``homology setup'' and the ``homotopy setup''.
In the former scenario, the transferred structure $\bfnu$ and the extension $\bff$
are built inductively via homological obstruction theory, so that the end result is non-canonical.
A prototype of transfer theorems of this kind was established by
T.~Kadeishvili in his seminal paper~\cite{kadeishvili:RMS80}. A very
general formulation \cite[Theorem~2]{pet} together with a historical account can be found
in the recent paper of D.\ Petersen. In that work, $f_1$ is assumed to induce a  \qi\ of certain hom 
complexes.  

\subsection{The homotopy setup} \label{sec:hmtpy-setup}
This is the formalism which we will use in the present work. It was thoroughly developed
in \cite{Markl:2006}, with special cases and partial results appearing earlier in
the work of M.\ Kontsevich and Y.\ Soibelman
\cite{kontsevich-soibelman:00}, and  S.\ Merkulov \cite{merkulov:98}. In this approach,
which is valid over an arbitrary commutative ring, a transfer exists provided that we have a left homotopy inverse $g_1$ to $f_1$, and a~chain homotopy 
$h \maps g_1 f_1 \simeq \id_{A}$, as in Eq.\ \ref{eq:htt}.

The homotopy setup, in fact, yields explicit formulas for the transfer and much more.
Fix a left homotopy inverse $g_1$ of $f_1$, and a chain homotopy $h$, as above. Then the formulas in \cite{Markl:2006} produce an {\em explicit\/} $A_\infty$-structure $(A',d',\bfnu) =
(A',d',\nu_2,\nu_3,\ldots)$ on $A'$, an {\em explicit\/} extension
$\bff \maps (A,d,\bfmu) \to (A',d',\bfnu)$
of the chain map $f_1$, as well as an {\em explicit\/}  extension 
\begin{equation}
\label{eq:g}
\bfg = (g_1,g_2,g_3,\ldots) \maps (A',d',\bfnu) \longrightarrow
(A,d,\bfmu) 
\end{equation}
of the chain map $g_1$, and an {\em explicit\/} extension $\bfh = (h,h_2,h_3,\ldots)$
of the homotopy $h$.  The extension $\bfg$ plays a crucial role in Sec.\ \ref{sec:class}, but $\bfh$ will not be needed.

We recall the formulas for the transferred structure $(A',d',\bfnu)$.  
According to the Ansatz \cite[Eq.\ 1]{Markl:2006}, the structure
operations $\nu_n$ are of the form
\begin{equation}
\label{eq:nu-n}
\nu_n := f_1 \circ \p_n \circ \otexp {g_1} n,\  n \geq 2,
\end{equation} 
where the {\bf p-kernels\/} \cite[Section~4]{Markl:2006} $\p_n \maps {A}^{\otimes n} \to A$ 
are defined as follows. Let $\Trp_n$ denote the set of planar rooted trees whose
vertices all have at least two incoming edges, with internal edges decorated by the symbol\krouzekstopka, and which have $n$ leaves. Elements of $\Trp_n$ encode
maps and their compositions. For example, the tree
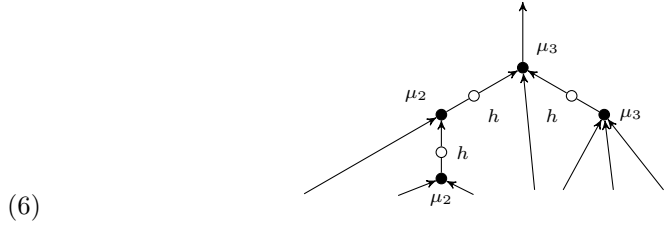
\vspace{-1ex}
\begin{equation} \label{eq:tree}
\begin{tikzpicture}
\tikzstyle{w}=[circle, draw, minimum size=4, inner sep=1]
\tikzstyle{b}=[circle, draw, fill, minimum size=4, inner sep=1]

\begin{scope}[yshift=0cm,scale=1.0, >=stealth']
\node[b][label=above right:\tlab{$\mu_3$}](b1) at (0,0) {};
\node[w][label=below right:\tlab{$h$}](w1) at (-150:0.75cm) {};
\node[w][label=below left:\tlab{$h$}](w2) at (-30:0.75cm) {};
\node[b][label=above left:\tlab{$\mu_2$}](b2) at (-150:1.25cm) {};
\node[b][label=right:\tlab{$\mu_3$}](b3) at (-30:1.25cm) {};

\node(r) at (0,1){};
\draw[->] (b1) edge (r);
\draw[->] (w1) edge (b1);
\draw[->] (w2) edge (b1);
\draw (b2) edge (w1);
\draw (b3) edge (w2);

\begin{scope}[shift={(-150:1.25cm)}]
\node(rr) at (0,0.05){};
\node[w][label=right:\tlab{$h$}](w3) at (-90:0.5cm) {};
\node[b][label=below:\tlab{$\mu_2$}](b4) at (-90:0.85cm) {};
\draw[->] (w3) edge (rr);
\draw (b4) edge (w3);
\end{scope}

\begin{scope}[shift={(-150:3.5cm)}]
\node (a) at (0,0){};
\node (b) at (1.25,0){};
\node (c) at (2.5,0){};
\node (d) at (3.2,0){};
\node (e) at (3.50,0){};
\node (f) at (4.25,0){};
\node (g) at (5.0,0){};

\end{scope}

\draw[->] (a) edge (b2);
\draw[->] (b) edge (b4);
\draw[->] (c) edge (b4);
\draw[->] (d) edge (b1);
\draw[->] (e) edge (b3);
\draw[->] (f) edge (b3);
\draw[->] (g) edge (b3);
\end{scope}
\end{tikzpicture}
\end{equation}
is an element of $\Trp_7$. We assign to every tree $T \in \Trp_n$ a map $F_T \maps \otexp {A}n
\to {A}$ such that each \krouzekstopka\ corresponds to the homotopy $h \maps A
\to A$, and each vertex with $k$ incoming edges corresponds to
the map $\mu_k \maps \otexp {A}k \to A$. For example, the tree $T$ in \eqref{eq:tree}
is assigned to the degree $5$ map $F_T = \mu_3(h \circ \mu_2(\id_{A} \ot \  
h \circ \mu_2) \ot \id_{A} \ot \ h \circ \mu_3) \maps \otexp {A}7 \to A$.
The p-kernels in \eqref{eq:nu-n} are then given by
\begin{equation*}
\p_n := \sum_{T \in \Trp_n} (-1)^{\vartheta (T)} \cdot F_T,
\mbox { $n \geq 2$.}
\end{equation*}
where the sign $(-1)^{\vartheta (T)}$ depends\footnote{We will not need the precise definition of $(-1)^{\vartheta (T)}$; see \cite[Prop.\ 6]{Markl:2006} for details.}  on the number of subtrees in $T$ of a certain type. Notice that, while $\nu_2 = f_1 \circ \mu_2 \circ (g_1 \ot g_1)$,  
the higher arity transfer operations $\nu_{n \geq 3}$ depend on the
homotopy $h$, as well as $f_1$ and $g_1$.

In the rest of the paper, $f_1$ will always be a chain homotopy
equivalence. We will call the explicit transfer given
by~\eqref{eq:nu-n} the {\bf transfer over a chain homotopy equivalence
  $f_1$,} in contrast to a less specific transfer over a chain map $f_1$ defined at the
beginning of this section. 

\section{Classifying transferred structures} \label{sec:class}
In this section, we present the main results previously summarized  in
Sec.\ \ref{sec:results}. In  \ref{sec:strict-iso}--\ref{sec:hmpty-cases}, 
we address the strict isomorphism, weak isomorphism, and homotopical variations of Question \ref{Q1}, giving us seven cases in total to consider. In Sec.\ \ref{sec:hmpty-cases}, we also address Question \ref{Q2}, the precise version of
D.\ Sullivan's original query. 

The starting point for all results in this section is an
$A_\infty$-algebra $(A,d,\bfmu)$, a~{\it chain homotopy equivalence}
$f_1 \maps (A,d) \to (A',d')$, and an $A_\infty$-algebra
$(A',d',\bfmu')$. The goal is to compare the latter $A_\infty$-algebra
to a transfer $(A',d',\bfnu)$ of the former over $f_1$ via the
``homotopy setup'' from Sec.\ \ref{sec:hmtpy-setup}.

\subsection{Strict isomorphism: Cases 1 and 2}  \label{sec:strict-iso}
These are the only variations of Question \ref{Q1} in which explicit formulas for the transfer matters.
We simply check whether the operations $\mu'_n$ of
$(A',d',\bfmu')$ are either: (1) equal to the operations $\nu_n$ defined via Eq.\ \ref{eq:nu-n}, or (2)
equal to a twist of these operations by the automorphism $\phi_1\maps (A',d') \xto{\cong} (A',d')$.

\begin{remark}
\label{sec:str-vs-quasi}
Characterizing transfers via strict isomorphism leads to an interesting side question which is also related to Thm.\ \ref{thm:isotopy} below. Suppose that we are given a weak $A_\infty$-morphism
$\bfF \maps (A,d,\bfmu) \to  (A',d',\bfmu')$ which extends a~{\it
  quasi-isomorphism} of chain complexes $f_1$. That is,
$(A',d',\bfmu')$ is a transfer of $(A,d,\bfmu)$ over a chain map
$f_1$. Can one enhance $f_1$ into a homotopy data such that $ (A',d',\bfmu')$
equals, or is strictly isomorphic to, the transfer \eqref{eq:nu-n} of $ (A,d,\bfmu)$
over a chain homotopy equivalence $f_1$?  

The answer is \underline{no} in general, as the following example shows.
Let  $(A,d,\bfmu)$ be the free associative $\eR$-algebra $ \eR\gen{x}$ generated by
an element~$x$ of degree~$0$.  Interpret  $ \eR\gen{x}$  as an $A_\infty$-algebra with
the trivial differential and all structure operations except $\mu_2$
trivial. Let $(A',d',\bfmu')$ be the free associative $\eR$-algebra
$\eR\gen{x,u,\overline u}$ generated by $x$ of degree $0$, $u$ of
degree $2$ and $\overline u$ of degree $1$, with the differential
given by $d' x = d' \overline u := 0$, and $d' u := \overline u$.
Finally, let $\bfF \maps \bigl(\eR\gen{x},d = 0 \bigr) \to \bigl(\eR\gen{x,\overline u,u},d'\bigr)$ 
be the dg algebra morphism $\bfF(x):=x$,
viewed as a strict $A_\infty$-morphism $\bfF = (f_1,0,0,\ldots)$
with $f_1 := \bfF$. 

Consider the possible left homotopy inverses $g_1$ of $f_1$. Since the
differential of $\eR\gen{x}$ is trivial, $g_1$ must be a strict
inverse, and we easily see that the only possibility is that
$g_1(x^k) := x^k$ for $k \geq 0$, while $g_1$ is trivial on the
remaining elements of $\eR\gen{x,\overline u,u}$. Moreover, the
homotopy $h$ witnessing $g_1\cc f_1 \simeq \id$ must be zero
for degree reasons. Hence, the formulas \eqref{eq:nu-n} for the transferred 
$A_\infty$-structure give us 
\[
\nu_2(a,b) :=
\begin{cases}
\mu'_2(a,b) & \hbox{if $a = x^k$, $b =x^l$ for some $k,l \geq 0$, and}
\\
0&\hbox{otherwise,}
\end{cases} 
\]
while $\nu_n := 0$ for $n \geq 3$. It is easy to check that this
transferred structure is neither equal to nor strictly isomorphic to $\eR\gen{x,\overline u,u}$.
\end{remark}

\subsection{Weak isomorphism: Cases 3 and 4} \label{sec:isotopy}
This is the most interesting variation of Question \ref{Q1}. The main technical tool used here is the obstruction theory developed in the next section (Sec.\ \ref{sec:obst}). We start with 
Case 3, which concerns the isotopy class of a transferred structure $(A',d',\bfnu)$.

\begin{theorem}
\label{thm:isotopy}
Given a chain homotopy equivalence $f_1 : (A,d) \to (A',d')$, an
$A_\infty$-structure $(A',d',\bfmu')$ on $(A',d')$ is a transfer over
the chain map $f_1$ if and only if it is isotopic to the transfer over the
chain homotopy equivalence $f_1$. 
\end{theorem}

\begin{proof}
Assume that $(A',d')$ is a transfer over the chain map $f_1$, i.e.\
that $f_1$ extends to a  weak $A_\infty$-morphism
$
\bfF =  (f_1,F_2,F_3,\ldots)\maps (A,d,\bfmu) \to (A',d',\bfmu')
$
and promote the chain homotopy equivalence $f_1$ to the data
\begin{equation}
\label{eq:hmtpy-data}
\xymatrix@1{     *{ \quad \ \  \quad (A, d)\ } \ar@(dl,ul)[]^{h}\ 
\ar@<0.5ex>[r]^{f_1} & *{\
(A',d') \quad \ \  \ \quad } \ar@(ur,dr)[]^{l}
\ar@<0.5ex>[l]^{g_1}
}\hskip -1em   ~  g_1 f_1 - \id_{A}  =d  h+ h  d, \
f_1 g_1 - \id_{A'}  =d'  l + l  d'.
\end{equation}
Let $(A',d',\bfnu)$ be the structure transferred over $f_1$ using
$f_1,g_1$ and $h$. Then $f_1$ can be extended to a weak $A_\infty$-morphism
$\bff \maps (A,d,\bfmu) \to (A',d',\bfnu)$ as in \eqref{eq:f} 
and $g_1$ can also be extended to $\bfg \maps  (A',d',\bfnu) \to (A,d,\bfmu)$
as in \eqref{eq:g}. The linear term $(\bfF\! \circ\! \bfg)_1$ of the composition
$\bfF \cc \bfg \maps (A',d',\bfnu) \to (A,d,\bfmu')$
equals $f_1 \circ g_1$, which is homotopic to the identity
$\id_{A'}$ via the homotopy $l$ in \eqref{eq:hmtpy-data}.
It then follows from Prop.\ \ref{prop:lift} in Sec.\ \ref{sec:obst} 
that there exists a weak $A_\infty$-morphism of the form 
\[
\bfphi:= (\id_{A'},\phi_2,\phi_3,\ldots) : (A',d',\bfnu) \to
(A',d',\bfmu').
\]
Hence, $\bfphi$ is our desired isotopy.

The opposite implication is simple. If $\bfphi \maps  (A',d',\bfnu) \to
(A',d',\bfmu')$ is an isotopy, then $\bfF := \bfphi \circ \bff$ is
a weak $A_\infty$-morphism extending $f_1$.
\end{proof}

\begin{corollary}\label{cor:isotopy}
The isotopy type of the transfer over a given chain homotopy
equivalence  $f_1$ does not depend on
the choices of $g_1$ and $h$ in \eqref{eq:htt}. 
If~$R$ is a field, then any $A_\infty$-structure which is a target
of an $A_\infty$-quism enhancing $f_1$ is isotopic to a structure transferred over 
a~chain homotopy equivalence enhancing~$f_1$.
\end{corollary}

In light of Thm.\ \ref{thm:isotopy}, we invite the reader to verify
that the two $A_\infty$-structures on the complex
$\eR\gen{x,\overline u,u}$ considered in Remark~\ref{sec:str-vs-quasi}
are indeed isotopic.
On the other hand, if $f_1 \maps 0 \to (A',d')$ is the trivial chain
map, then the class of transfers over $f_1$ consists of all
$A_\infty$-structures on $(A',d')$. 
Hence, 
the class of transfers over a chain map will not equal the isotopy class of a given transfer, in general.

The analogous result for Case 5, which involves weak isomorphism classes~is:   
\begin{theorem}\label{thm:weak-iso}
Given a chain homotopy equivalence $f_1 : (A,d) \to (A',d')$, an
$A_\infty$-structure $(A',d',\bfmu')$  is  a transfer of
$(A,d,\bfmu)$ over the chain map $\phi_1 \cc f_1$ for some 
automorphism $\phi_1 \maps (A',d') {{\xto{\cong}}}(A',d')$
if and only if it is weakly isomorphic
to the transfer of  $(A,d,\bfmu)$ over the chain homotopy
equivalence $f_1$.
\end{theorem}
The proof is a simple modification of the one given for  
Thm.\ \ref{thm:isotopy}, so we omit it.

\subsection{The homotopical Cases 5 and 6} \label{sec:hmpty-cases} 
To give sensible answers for these cases, we assume that $R$ is a field, so that $f_1$ is a chain homotopy equivalence if and only if it is a \qi\ of complexes.
It then turns out that Cases  (5a), (5b), and (6) are equivalent. As before, let $(A',d',\bfnu)$ denote a
transfer of $(A,d,\bfmu)$ over the quasi-isomorphism of complexes $f_1$,
and let  $(A',d',\bfmu')$ be an arbitrary $A_\infty$-structure on $(A',d')$. Below, the relation ``$ \simeq $'' denotes weak equivalence of $A_\infty$-algebras.
\begin{proposition}
\label{prop:weq}
The following six conditions are equivalent:\\
\setlength{\tabcolsep}{1.75pt}
\begin{tabular}{l l} 
(i) $\exists\, A_\infty$-quism $(A',d',\bfnu) \to (A',d',\bfmu')$,  & (iv) 
$\exists\, A_\infty$-quism $(A',d',\bfmu') \to  (A,d,\bfmu)$,\\
(ii) $\exists\, A_\infty$-quism  $(A',d',\bfmu') \to (A',d',\bfnu)$, & 
(v) $\exists\, A_\infty$-quism  $(A,d,\bfmu) \to (A',d',\bfmu')$,\\
(iii) $(A',d',\bfnu) \simeq (A',d',\bfmu')$, & (vi) $(A,d,\bfmu) \simeq (A',d',\bfmu')$.\\
\end{tabular}
\end{proposition}

We need the following lemma, which is \cite[Theorem~10.4.4]{LV} with
${\cP}$ the operad for associative algebras; it also follows from
an abstract homotopy theoretic argument
\cite[Sec.\ 3.7]{Keller}.

\begin{lemma}
\label{lem:quasi-inv}
Let $(B,d,\bfomega)$ and
$(B',d',\bfomega')$ be  $A_\infty$-algebras. Then there exists an $A_\infty$-quism $\bfalpha : (B,d,\bfomega) \to  (B',d',\bfomega')$ if and only if 
there exists an $A_\infty$-quism $\bfbeta : (B',d',\bfomega') \to  (B,d,\bfomega)$ in the opposite
direction.
\end{lemma}


\begin{proof}[Proof of Prop.\ \ref{prop:weq}]
Since $f_1$ is a quasi-isomorphism of chain complexes by assumption,
$(A,d,\bfmu)$ is weakly equivalent to its transfer
$(A',d',\bfnu)$. Weak equivalence is an equivalence relation,
therefore (iii) is equivalent to (vi). On the other hand, by definition, a weak equivalence is a zig-zag of $A_\infty$-quisms, and each arrow of this zig-zag can be inverted by Lemma \ref{lem:quasi-inv}. 
This makes the remaining equivalences clear.
\end{proof}

\subsection{Necessary and sufficient conditions for isotopy} \label{sec:Q2}
We now address Question \ref{Q2}, the precise version of
D.\ Sullivan's original query. It is motivated by an observation concerning the characterization of weak isomorphism classes, based
on a conjecture  communicated to the
authors by Sullivan. (Recall that isotopy is a special case of a weak isomorphism.)  A form of this conjecture is proven as  Thm.~\ref{thm:Sull}
below. Corollary \ref{cor:iso-char} is then our answer to Question \ref{Q2}
given in the language of obstruction theory.

Let $\setinf(A,d)$ denote the set of weak
isomorphism classes of $A_\infty$-structures on a given chain complex
$(A,d)$. Following \cite[Sec.\ 6]{Markl:2006}, since $f_1$ is assumed to be a chain homotopy equivalence, choosing homotopy data as in \eqref{eq:hmtpy-data} induces maps of sets
\begin{equation}
\label{eq:biject}
\Trans fgh \maps \setinf(A,d) \to \setinf(A',d'), \quad \Trans gfl \maps \setinf(A',d') \to \setinf(A,d).
\end{equation}
Proposition 10 in \cite{Markl:2006} implies that the functions $\Trans fgh$ and
$\Trans gfl$ are mutually inverse bijections. 
In the theorem below, if $(A,d)$ is a subcomplex of $(A',d')$, then an {\bf  extension\/} of an
$A_\infty$-structure  $\bfmu = (\mu_2,\mu_3,\ldots)$ on $(A,d)$ to $(A',d')$ is an $A_\infty$-structure $\bfnu =(\nu_2,\nu_3,\ldots)$ on $(A',d')$ such 
that the restriction $\nu_n|_{A^{\ot n}}$ equals $\mu_n$ for  $n \geq 2$.
Note that we formulate the second statement of the theorem in terms of weak isomorphism classes
since we do not know whether bijections analogous to those in~(\ref{eq:biject}) exist for
isotopy classes.

\begin{theorem} \label{thm:Sull}
Suppose $R$ is a field, and that $(A,d)$ is a subcomplex of $ (A',d')$ such that 
the inclusion $\iota: 
(A,d) \hookrightarrow (A',d')$ is a quasi-isomorphism. Then
\begin{enumerate}
\item
The isotopy class of a transfer of
an $A_\infty$-structure $(A,d,\bfmu)$ over the chain map 
$\iota$ contains an extension of the family
$\bfmu = (\mu_2,\mu_3,\ldots)$ 
to $A'$.
\item
Moreover, the $A_\infty$-structure  $(A,d,\bfmu)$ is
characterized, up to weak isomorphism, by the weak
isomorphism class of its extension.   
\end{enumerate}
\end{theorem}

\begin{proof}
Since we are working over a field, we may promote the initial setup
into the data in \eqref{eq:hmtpy-data}, with $f_1 := \iota$, $g_1$
a strict left inverse $\pi$ of $\iota$, $h := 0$ and $l$ an arbitrary
chain homotopy between $\iota \pi$ and $\id_{A'}$. 
Formulas \eqref{eq:nu-n} then clearly determine the pieces $\nu_n$, $n \geq 2$, of the transferred
structure as the extensions $\nu_n: = \iota \circ \mu_n \circ \pi^{\ot n}$
of $\mu_n$. Part (1) then follows from Theorem~\ref{thm:isotopy}.

If the same weak isomorphism class of $A_\infty$-structures on
$(A',d')$ contains extensions of two $A_\infty$-structures on
$(A,d)$, then these structures must be weakly isomorphic since
the maps \eqref{eq:biject} of weak isomorphism classes are bijections. This proves
part~(2) of the theorem. 
\end{proof}

Returning to the general situation over an arbitrary commutative ring, Thm.\ \ref{thm:isotopy} combined
with Cor.\ \ref{cor:Ainf-obs} below provides a characterization of transfers up to isotopy.

\begin{corollary}\label{cor:iso-char}
The obstruction to exhibit an isotopy between $(A',d',\bfmu')$ and the transfer
of  $(A,d,\bfmu)$ over a chain homotopy equivalence $f_1 \maps (A,d) \to (A',d')$
is an infinite sequence of homology classes determined by $\bfmu$, $f_1$, and $\bfmu'$:
\[
\Bigl\{[\kappa_n] \in H_{n-2} \bigl( \Hom_R(A^{\tensor n},A')\bigr) ~ \vert ~ n \geq 2 \Bigr\}
\]
where the differential on the complex $\Hom_R(A^{\tensor n},A')$ is the canonical one induced by $d$ and $d'$.
\end{corollary}

\section{Obstruction theory for $A_\infty$-morphisms} \label{sec:obst}
We develop in this section the tools needed to prove Thm.\ \ref{thm:isotopy}, Thm.\ \ref{thm:weak-iso}, and Cor.\ \ref{cor:iso-char}. We begin by recalling some basic facts concerning dg coalgebras and $A_\infty$-algebras, following \cite[Sec.\ 2]{Kopriva} and \cite[Sec.\ 1.26]{LV}.
\subsection{Coalgebras and the bar construction} \label{sec:coalg} Let
$V$ be a graded $R$-module. We denote by
$\bigl(\T^c(V),\rDelta \bigr)$ the reduced cofree conilpotent
coassociative coalgebra generated by $V$. Recall that this is the
graded coalgebra with underlying $R$-module
$\bigoplus_{n \geq 1}V^{\tensor n}$ equipped with the comultiplication
$
\rDelta(\simten{v}{n}):= \sum_{i=1}^{n-1} (\simten{v}{i}) \bigotimes
(v_{i+1} \tensor \cdots \tensor v_{n}).
$ 
We denote by
$\rdDelta{n} \maps \Tc(V) \to \Tc(V)^{\tensor n+1}$ the $n$th reduced
diagonal: the $R$-linear map defined recursively as
$\rdDelta{0}\!:=\id$, $\rdDelta{1}\!:=\rDelta$, and
\hbox{$\rdDelta{n}\!  : =(\rDelta \!\tensor\! \id^{\tensor (n-1)})\! \circ\!
\rdDelta{n-1}$} for $n > 1$. By construction, for $k < n$ we have
\begin{equation} \label{eq:ker-diag}
v_1 \tensor v_2 \tensor \cdots \tensor v_k \in \ker \rdDelta{n-1} \quad \forall v_1,\ldots,v_k \in V.
\end{equation}

Given a linear map $F \maps \Tc(V) \to \Tc(W)$ and integers $m,n \geq 1$, we denote by $F^m_n \maps V^{\tensor n} \to W^{\tensor m}$ the restriction $F \vert_{V^{\tensor n}}$ composed with the projection $\Tc(W) \to W^{\tensor m}$. In addition, linear maps corresponding to elements of the graded $R$-module 
\begin{equation} \label{eq:homcmplx}
\Hom_R(\Tc(V),W) \cong \prod_{n \geq 1} \Hom_R(V^{\tensor n},W)
\end{equation}
will be denoted as $F^1=(F^1_1,F^1_2,\cdots)$. Recall that there is a one-to-one correspondence  \cite[Sec.\ 2.1]{Kopriva} between degree $-1$ linear maps $D^1 \in \Hom_R(\Tc(V),V)$ and degree $-1$ coderivations $D \maps \Tc(V) \to \Tc(V)$ 
given explicitly by 
\begin{equation} \label{eq:coder}
D^{m}_n:=\sum_{ \substack{i +j = m-1\\ i, j \geq 0}} \id^{\tensor i} \tensor D^1_{n- m + 1} \tensor \id^{\tensor j}
\end{equation} 
for each $n\geq 1$. Note that $D^{m}_n=0$ if $m > n$.  A {\bf codifferential} on $\Tc(V)$
is a~degree~$-1$ coderivation $D$ as above satisfying $D\circ D=0$, or equivalently, 
for all $n \geq 1$:
\begin{equation} \label{eq:codiff}
\sum_{k =1}^n D^1_k \circ D^k_n =0.
\end{equation} 

Analogously, there is a one-to-one correspondence \cite[Sec.\ 2.2]{Kopriva} between degree~$0$ linear maps $F^1 \in \Hom_R(\Tc(V),V')$ and coalgebra morphisms $F \maps \Tc(V) \to \Tc(V')$, given explicitly by the formulas 
\begin{equation} \label{eq:comorph}
F^{m}_n:= \sum_{i_1 + i_2 + \cdots + i_m =n} F^1_{i_1} \tensor F^{1}_{i_2} \tensor \cdots \tensor F^1_{i_m},  
\end{equation} 
for each $n\geq 1$. In particular, $F^m_n=0$ if $m > n$. If $D$ and $D'$ are codifferentials on $\Tc(V)$ and $\Tc(V')$, respectively, then a coalgebra morphism $F \maps \Tc(V) \to \Tc(V')$ satisfies $D' \cc F = F \cc D$ if and only if
for all $n \geq 1$:
\begin{equation} \label{eq:dgcomorph}
\sum_{k=1}^n {D^{'}}^{1}_k \cc F^k_n = \sum_{k=1}^n F^1_k \cc D^k_n.
\end{equation}
In this case, $F \maps (\Tc(V),D) \to (\Tc(V'),D')$ is a {\bf morphism} of dg-coalgebras.


\subsubsection{The bar construction} \label{sec:bar}
Lastly, we recall the functorial assignment of an $A_\infty$-algebra $(A,d,\bfmu)$ 
to the coalgebra $C(A):=\Tc(\ua  A)$ equipped with the codifferential $\dlt \maps C(A) \to C(A)$ 
defined as $\dlt^1_1:= \ua  \cc d \cc \da$, and $\dlt^1_{n}:= \ua  \cc \mu_n \cc (\da)^{\tensor n}$, for $n \geq 2$. The assignment is fully faithful: there is a one-to-one correspondence \cite[Sec.\ 2.3]{Kopriva} between weak $A_\infty$-morphisms $\bff \maps 
(A,d,\bfmu) \to (A',d',\bfmu')$ and dg coalgebra morphisms $F \maps (C(A),\dlt) \to (C(A'),\dlt')$
given by the formulas $F^1_n:= \ua  \cc f_n \cc (\da)^{\tensor n}$, for all $n \geq 1$.
In what follows, $C^{n}(A)$ and $C^{\leq n}(A)$ denote the graded $R$-modules $(\ua  A)^{\tensor n}$ and $\bigoplus_{k \geq 1}^n (\ua  A)^{\tensor k}$, respectively.

\subsection{Operations on the Hom complex}
Let $(A,d,\bfmu)$ and $(A',d',\bfmu')$ be $A_\infty$-algebras; let $\bigl(C(A),\dlt \bigr)$ and
$\bigl(C(A'),\dlt' \bigr)$ denote their corresponding dg coalgebras. Consider the graded $R$-module $\cM:=\Hom_{R}\bigl(C(A),\ua  A' \bigr)$, as defined in \eqref{eq:homcmplx},
equipped with the differential
\begin{equation*} 
\del F^1:= {\dlt^{'}}^{1}_1 \cc F^1 - (-1)^{m} F^1 \cc \dlt
\end{equation*}
where $F^1 \maps C(A) \to \ua  A'$ is a degree $m$ $R$-linear map. Observe that $(\cM,\del)$ admits a~descending filtration of dg submodules $\cM=\cF_1 \cM \supseteq \cF_2 \cM \supseteq \cdots$
\[
\cF_r\cM:= \bigl \{ F^1 \in \Hom_R(C(A),\ua  A') ~ \vert ~  F^1 \vert_{C^{\leq r-1}(A)} =0 \bigr\}.
\]
Via the isomorphisms
\[
\cF_{r-1} \cM/\cF_{r} \cM \cong \Hom_R(C^{r-1}(A),\ua  A') \hbox{ and } 
\cM/\cF_{r} \cM \cong 
 \Hom_R(C^{\leq r-1}(A),\ua  A'),
\] it is easy to see that $(\cM,\del)$ is {\bf complete} with respect to the topology induced by above filtration, i.e.\ $\cM \cong \plim_r \cM/\cF_r \cM$. 

\subsubsection{A codifferential on $\Tc(\cM)$} 
Given elements 
$\Find{F}{1}, \Find{F}{2},\ldots,\Find{F}{n} \in \cM$, 
let $\Find{F}{1} \btensor \Find{F}{2} \btensor \cdots \btensor \Find{F}{n} \in \cM^{\tensor n}$ denote the usual corresponding tensor\footnote{The notation $\Find{F}{i}$ 
should not be confused with $F^1_i$, i.e.\ the $i$th component of an element $F^1=(F^1_1,F^1_2,F^1_3,\ldots) \in \cM$.}. In particular, we denote by ${F^1}^{\, \btensor n}$  $n$-fold tensor product of $F^1 \in \cM$.

The next result concerns the properties of the linear maps $\cQ^1_n \maps \cM^{\tensor n} \to \cM$ defined as $\cQ^1_1(F):=\del F$, and for $n\geq 2$
\begin{equation} \label{eq:Q}
\cQ^1_{n}\bigl(\Find{F}{1} \btensor \Find{F}{2} \btensor\cdots \btensor\Find{F}{n} \bigr):=
{\dlt'}^1_n \cc \bigl(\Find{F}{1} \tensor \Find{F}{2} \tensor \cdots \tensor \Find{F}{n}) \cc \rdDelta{n-1}.
\end{equation}
Note that \eqref{eq:ker-diag} implies that the maps $\cQ^1_{n \geq 2}$ are compatible with the filtration on $\cM$, i.e.\
\begin{equation} \label{eq:Q-filt}
\cQ^1_n \bigl(\cF_{j_1}\cM,\cF_{j_2}\cM, \cdots, \cF_{j_n}\cM \bigr) \sse \cF_{j_{1} + j_{2} + \cdots + j_{n}} \cM. 
\end{equation}

A variation of following lemma was given in \cite[Sec.\ 4]{deKleijn} and \cite[Sec.\ 7.2]{deKleijn} for the case when $R$ is a field.
\begin{lemma}\label{lemma:Ainf-hom}
\mbox{} 
\begin{enumerate}[leftmargin=18pt]
\item The linear maps $\{\cQ^1_n\}_{n \geq 1}$ induce, via the formulas \eqref{eq:coder}, a degree $-1$ codifferential $\cQ$ on the coalgebra $\Tc(\cM)$. 

\item Given a degree 0 element $F^1 \in \cM$, the assignment 
\begin{equation} \label{eq:curv}
F^1 ~ \mapsto ~ \cR(F):= \sum_{n=1}^{\infty} \cQ^1_n( {F^1}^{\, \btensor n}) \in \cM_{-1}
\end{equation}
induces a well-defined set-theoretic function $\cR \maps\! \cM_0\! \to\! \cM_{-1}$.\ Moreover, \hbox{$\cR(F)=0$} if and only if $F^1$ corresponds, via the formulas \eqref{eq:comorph}, to a dg coalgebra morphism $F \maps (C(A),\dlt) \to (C(A'),\dlt')$

\item For all $F^1 \in \cM_0$, the following identity holds:
\begin{equation} \label{eq:bianchi}
\cQ^1_1 \cR(F) + \sum_{n=2}^{\infty} \sum_{k=0}^{n-1}\cQ^1_n \Bigl( {F^1}^{\, \btensor (n-1)-k} \btensor \cR(F) 
\btensor {F^1}^{\, \btensor k} \Bigr) = 0.
\end{equation}
\end{enumerate}
\end{lemma}
\begin{proof}
(1) Note that $\cQ^1_1 \cc \cQ^1_1 =0$, since $\cQ^1_1=\del$ is a differential. Let $n > 1$. Since $\dlt'$ is a codifferential on $C(A')$, we have $\sum_{k=1}^n {\dlt'}^1_k \cc {\dlt'}^k_n=0$, and the coLeibniz rule implies that $\rdDelta{n-1} \cc \delta' = \sum_{i=1}^{n}( \id^{i-1} \tensor \delta' \tensor \id^{n-i}) \cc \rdDelta{n-1}$. A direct computation using these equalities, along with Eq.\ \ref{eq:coder}, gives $\sum_{k=1}^{n} \cQ^1_k \cc \cQ^k_n =0$. 

(2) Since $\cM = \cF_1 \cM$, Eq.\ \ref{eq:Q-filt} implies that 
$\cQ^1_n({F^1}^{\, \btensor n}) \in \cF_n \cM_{-1}$ for all \hbox{$n \geq 1$}. Hence, the infinite summation in the definition of $\cR(F)$ converges, since $\cM$ is complete, and so  $\cR \maps \cM_0 \to \cM_{-1} $ is a well-defined function. From combining Eq.\ \ref{eq:comorph} and Eq.\ \ref{eq:Q} along with the identity $F^m_n= (F^1)^{\tensor m} \cc \rdDelta{m-1} \vert_{C^n(A)}$, it follows that $F^1$ is in the zero locus of $\cR$ if and only if the corresponding coalgebra map $F$ satisfies Eq.\ \ref{eq:dgcomorph}.

(3) Let $F^1 \in \cM_{0}$. The left-hand side of Eq.\ \ref{eq:bianchi} is a sum of terms of the form
$s_{m,\el}:=  \sum_{k=0}^{m-1} {F^1}^{\btensor (m-1)-k} \btensor \cQ^1_\el( {F^1}^{\, \btensor \el} ) \btensor {F^1}^{\btensor k}$
for $m, \el \geq 1$. From  Eq.\ \ref{eq:coder}, we deduce that $s_{m,\el} = \cQ^1_{m} \cc \cQ^{m}_{m + \el -1} \bigl( {F^1}^{{\, \btensor m + \el -1}} \bigr)$. The desired equality \eqref{eq:bianchi} will then follow from Eq.\ \ref{eq:codiff}, or equivalently, the fact that $\cQ \cc \cQ =0$.
\end{proof}

The next proposition shows that $(\Tc(\cM),\cQ)$ encodes the obstruction theory for dg coalgebra morphisms between $(C(A),\dlt)$ and $(C(A'),\dlt')$. 
\begin{proposition}\label{prop:obs}
Let $m >1$ and suppose $\{F^1_1,\ldots F^1_{m-1}\}$ is a collection of degree~$0$ linear maps $F^1_k \maps C^k(A) \to \ua  A'$ such that the corresponding coalgebra
morphism $F \maps C(A) \to C(A')$ satisfies 
\[
(\dlt' \cc F - F \cc \dlt) \vert_{C^{\leq m-1}(A)} =0.
\]
Then the linear map $c_m(F) \maps C^m(A) \to \ua  A'$ defined as
\begin{equation} \label{eq:obs}
c_m(F):=  \sum_{k=2}^m \dlt^{\prime 1}_k \cc F^k_m - \sum_{k=1}^{m-1} F^1_k  \cc \dlt^k_m
\end{equation}
is a degree $-1$ cycle in the quotient  \[
\bigl(\Hom_R  (C^m(A), \ua  A' ), \ba{\pa} \bigr)
\cong (\cF_m \cM, \pa)/(\cF_{m+1} \cM, \pa).
\] 
Moreover, there exists a linear map $\ti{F}^1_m \maps C^m(A) \to \ua  A'$ such that the coalgebra morphism $\ti{F} \maps C(A) \to C(A')$  corresponding to the collection $\{F^1_1,\ldots F^1_{m-1},\ti{F}^1_m\}$ satisfies 
\[
(\dlt' \cc \ti{F} - \ti{F} \cc \dlt) \vert_{C^{\leq m}(A)} =0
\]
if and only if $c_m(F)=-\ba{\pa} \ti{F}^1_m$. 
\end{proposition}
\def\squeezeddots{{\mbox{$. \hskip -.8pt . \hskip -.8pt .$}}}

\begin{proof}
The definition of the differential $\pa = \cQ^1_1$ on $\cM$ implies that we may write the induced differential on the quotient as 
$\ba{\pa} \ti{F}^1_m = (\cQ^1_1 \ti{F}^1_m) \vert_{\cC^m(A)} = \dlt^{\prime 1}_1 \circ \ti{F}^1_m  - \ti{F}^1_m \circ \dlt^m_m$, 
for any degree 0 map  
$\ti{F}^1_m \in \Hom_R  (C^m(A), \ua  A')$.  Hence, the second statement of the proposition follows directly from Eq.\ \ref{eq:dgcomorph}. It then remains to show that $\ba{\pa} c_m(F)=0$, or equivalently, that $\cQ^1_1 c_m(F) \in \cF_{m+1}\cM$.

Let $F^1=(F^1_1,F^1_2,\cdots,F^1_{m-1},0,0 \cdots) \in \cM_{0}$. For
$\el \geq 1$, let $\cR(F)_{\el}:=\cR(F) \vert_{C^{\el}(A)}$ denote the
restriction of the map \eqref{eq:curv} to the submodule
$C^{\el}(A)$. Then $\cR(F)_{\el}= \sum_{k=1}^{\el} \bigl( {\dlt'}^1_k
\cc F^k_\el - F^1_k \cc \dlt^k_{\el} \bigr)$. Note that $F^1_k$ makes
no contribution to $\cR_\el(F^1)$ if $k > \el$. Hence, the hypothesis
for the collection $\{F^1_1,\squeezeddots , F^1_{m-1}\}$ implies that \hbox{$\cR(F)_{\el \leq m-1} =0$}, and so we have $\cR(F) \in \cF_{m}\cM$. Since the linear maps $\{\cQ^1_n\}$ are compatible with the filtration on $\cM$, it follows from Eq.\ \ref{eq:bianchi} that $\cQ^1_1 \cR(F) \in \cF_{m+1}\cM$. On the other hand, $\cR(F)_m  = c_m(F) \in \cF_{m}\cM$, since the $k$th components of $F^1$ vanish for $k \geq m$. Therefore, $\cR(F) - c_m(F) \in \cF_{m+1}\cM$, and so we conclude that $\cQ^1_1 c_m(F) \in \cF_{m+1}\cM$.   
\end{proof}
Using, for each $n \geq 2$, the $R$-module isomorphisms $\Hom_R(C^n(A),\ua  A')_{-1}  \cong \Hom_R(A^{\tensor n},A')_{n-2}$, we obtain as a corollary the obstruction theory for weak $A_\infty$-morphisms.
\begin{corollary} \label{cor:Ainf-obs}
The obstruction to lifting a chain map $f_1 \maps (A,d) \to (A',d')$ to a weak $A_\infty$-morphism
$\bff=(f_1,f_2,f_3,\ldots) \maps (A,d,\bfmu) \to (A',d',\bfmu')$ is an infinite sequence of homology classes $$\bigl\{[\kappa_n] \in H_{n-2} \bigl( \Hom_R(A^{\tensor n},A')\bigr) ~ \vert ~ n \geq 2 \bigr\}$$ where the differential on the hom complex $\Hom_R(A^{\tensor n},A')$ is the canonical one induced by $d$ and $d'$.
\end{corollary}

\subsection{Algebraic models for the interval} \label{sec:interval}
We will need a notion of homotopy between morphisms of dg coalgebras. 
\begin{definition} \label{def:interval}
A unital dg associative $R$-algebra $(\J,d, \cdotp)$ is a {\bf model for the interval} if there exists unital dg algebra morphisms $\eps_0,\eps_1 \maps \J \to R$, and $\jm \maps R \to \J$ such that
\begin{enumerate}[leftmargin=23pt]
\item The composition $R \xto{\jm} \J \xto{(\eps_0,\eps_1)} R \times R$ is the diagonal.
\item As chain maps, the morphisms $R \xto{\jm} \J \xto{\eps_i} R$
are deformation retractions for $i=0,1$.
\item Given two maps between chain complexes $f, g \maps (V,d_V) \to (W,d_W)$, and a~chain homotopy 
  between them, there exists a corresponding chain map $\ti{h} \maps V \to W \tensor_R \J$ such that $(\id \tensor \eps_0) \cc \ti{h} = f$ and $(\id \tensor \eps_1) \cc \ti{h} =g$.
\end{enumerate}  
\end{definition}
We recall two examples. The first is the normalized cochain algebra $N(I):=(N(\Del^1),d_N, \cup)$
on the 1-simplex with coefficients in $R$. As a graded $R$-module, $N(I)_{-1}:= R \vphi_{I}$, and  $N(I)_{0}:= R \vphi_0 \oplus R \vphi_1$. The differential is $d_N \vphi_0 := \vphi_I$, and  $d_N \vphi_1 := -\vphi_I$, and $\cup$ denotes the usual cup product. 
The following lemma is well known; the proof follows from a straightforward verification, so we omit it.
\begin{lemma} \label{lem:models}
The dg $R$ algebra $N(I)$ is an algebraic model for the interval over $R$ when equipped with the morphisms $\jm \maps R \to N(I)$, and  $\eps_0,\eps_1 \maps N(I) \to R$  defined as:
$\jm(1_R):= \vphi_0 + \vphi_1$, $\eps_0(\vphi_0):= \eps_1(\vphi_1):=1_R$, and $\eps_0(\vphi_1):= \eps_1(\vphi_0):=0$.
\end{lemma}
The second example is a graded commutative model for the case when \hbox{$R=\kk$} is a field of characteristic zero. This will be used in Sec.\ \ref{sec:Pinf}.
We denote by $\Om(I):=(\kk[t,dt], d_{\dR}, \wedge)$ the polynomial de
Rham algebra on the 1-simplex. As a graded vector space,
$\Om(I)_{-1}:= \kk[t]dt$, and $\Om(I)_0= \kk[t]$. The differential is 
\[
d_{\dR}(f(t) + g(t)dt) = \frac{df}{dt} dt,
\] 
and $\wedge$ is the usual wedge product. The obvious analog of Lemma \ref{lem:models} holds for $\Om(I)$, in which $\eps_0,\eps_1 \maps \Om(I) \to \kk$ are the evaluation maps at $t=0$, and $t=1$, respectively.

\subsubsection{Tensoring $A_\infty$-algebras with dg algebras} \label{sec:tensor}
Recall that if $(A,d,\bfmu)$ is an $A_\infty$-algebra, and $(B,d_B,\cdotp)$ is a dg associative algebra, then the tensor product $(A \tensor_R B, d_{\tensor}, \bfmu_{\tensor})$ is an $A_\infty$-algebra
with $d_\tensor:= d \tensor \id + \id \tensor d_B$, and
$${\mu_{\tensor}}_{k}(x_1 \tensor b_1, x_2 \tensor b_2, \cdots, x_k \tensor b_k):=
(-1)^{\varepsilon}\mu_k(x_1,\ldots,x_k) \tensor b_1 \cdotp b_2
\cdotp \cdots \cdotp b_k$$ 
for $k \geq 2$,
where $(-1)^\varepsilon$ is the usual Koszul sign. 
Note that this construction is functorial: if $\phi \maps B \to B'$ is a morphism of dg algebras then  
$\id_A \tensor \phi$ is a strict morphism of $A_\infty$-algebras.

\subsection{Lifting chain maps to weak $A_\infty$-morphisms} \label{sec:lift}
A special case of the proposition below, valid when $R$ is a field of characteristic zero, was given in \cite[Prop.\ 35]{Markl:2004} using different methods.
In what follows, $\bfT=(\tha_1,\tha_2,\cdots) \maps  (A,d,\bfmu) \to (A',d',\bfmu')$ is a given weak $A_\infty$-morphism, and $\Tha \maps (C(A),\dlt)  \to (C(A'),\dlt')$ denotes the corresponding morphism of dg coalgebras.
\begin{proposition}\label{prop:lift}
Suppose $\psi \maps (A,d) \to (A',d')$ is a chain map that is chain homotopic to $\tha_1 \maps (A,d) \to (A',d')$. Then there exists a weak $A_\infty$-morphism  $\bfpsi=(\psi_1,\psi_2,\cdots) \maps (A,d,\bfmu) \to (A',d',\bfmu')$ such that $\psi_1=\psi$.

\end{proposition}

\newcommand{\Lam}{\Lambda}
\newcommand{\pz}{{(0)}}
\newcommand{\po}{{(1)}}
\newcommand{\pii}{{(i)}}

\begin{proof}
Let $(\J,d_\J,\cdotp)$ be an algebraic model of the interval with 
$\eps_\pz,\eps_\po \maps \J \to R$, $\jm \maps R \to \J$ as described in 
Def.\ \ref{def:interval}. 
Let $J \maps (C(A'),\dlt') \to (C(A' \tensor \J), \dlt'_\tensor)$ 
and $E_{(i)} \maps (C(A' \tensor \J), \dlt'_\tensor) \to (C(A'),\dlt')$ for $i=0,1$,
denote the dg coalgebra morphisms corresponding to the 
strict $A_\infty$-morphisms $\id_{A'} \tensor \jm$, and $\id_{A'} \tensor\eps_{(i)}$ respectively.

To prove the proposition, we will use the obstruction theory developed in Prop.~\ref{prop:obs}
to inductively construct a dg coalgebra morphism $\mathsf{H} \maps (C(A),\dlt) \to (C(A' \tensor \J),\dlt'_\tensor)$ such that the linear component $\Psi^1_1$ of the composition $\Psi:=E_{(1)} \cc \mathsf{H} \maps (C(A),\dlt) \to (C(A'),\dlt')$ satisfies $\Psi^1_1 =\ua  \cc \psi \cc \da$.

For the base case, let $h \maps A \to \da A'$ be a chain homotopy satisfying $\psi -\tha_1 = d'h + h d$. Let $\ti{h} \maps (A,d) \to (A' \tensor \J,d'_\tensor)$ be the corresponding chain map as in Def.~\ref{def:interval}, and denote by $H \maps C(A) \to C(A'\tensor \J)$ the coalgebra morphism associated to the linear map $H^1_1:= \ua  \cc \ti{h} \cc \da$. Then by construction
\[
(\dlt'_{\tensor} \cc H - H \cc \dlt) \vert_{C^{1}(A)} =0, \quad (E_\po \cc H)^1_1 = \ua  \cc \psi \cc \da, \quad (E_\pz \cc H)^1_1 = \Tha^1_1.
\]

Now the inductive step. Let $m \geq  2$. Suppose $H \maps C(A) \to C(A' \tensor \J)$ is a~coalgebra morphism such that
\begin{equation} \label{eq:ind}
\begin{split}
&(\dlt'_\tensor \cc H - H \cc \dlt) \vert_{C^{\leq m-1}(A)} =0, \quad (E_\po \cc H)^1_1 = \ua  \cc \psi \cc \da, \\
&(E_\pz \cc H)^1_k = \Tha^1_k \qquad \text{for $k=1,\ldots, m-1$.}
\end{split}
\end{equation}

Consider the cycle $c_m(H) \in \bigl(\Hom_R\bigl(C^m(A), \ua (A \tensor \J)\bigr), \ba{\pa} \bigr)$ as defined in Eq.\ \ref{eq:obs}. We will show that it is a boundary. 
Composition with $E^1_{\pz 1}$ and $J^1_1$ gives chain maps
\[
\begin{split}
E^1_{\pz \ast} & \maps \Hom_R\bigl(C^m(A), \ua (A' \tensor \J)\bigr) \to 
\Hom_R\bigl(C^m(A), \ua  A'\bigr),\\
J^1_\ast &\maps \Hom_R\bigl(C^m(A), \ua  A'\bigr) \to  \Hom_R\bigl(C^m(A), \ua (A' \tensor \J)\bigr),
\end{split}
\]
respectively. Since $E_{\pz}$ corresponds to a strict $A_\infty$-morphism,  ${E_{(0)}}^{1}_{k}=0$ for $k \geq 2$.
Therefore, it follows from the induction hypothesis \eqref{eq:ind} and the definition of $c_m(H)$ that $E^1_{\pz \ast} (c_m(H)) = c_m(\Tha)$. Since $\Tha$ is a dg coalgebra morphism, the cycle $c_m(\Tha)$ is a boundary. In particular, $c_m(\Tha)= - \ba{\pa} \Tha^1_m$.
Since $\J$ is a model for the interval, and tensor product preserves chain homotopy equivalence, there exists a chain homotopy $\lambda \maps \ua  (A' \tensor \J) \to (A' \tensor \J)$ such that the chain maps $J^1_1 $ and $E^1_{\pz 1}$ 
satisfy $J^1_1 \cc E^1_{\pz 1} - \id_{\ua  (A' \tensor \J)} = {\dlt'}^1_{\tensor 1} \cc \lambda + \lambda \cc {\dlt'}^{1}_{\tensor 1}$, in addition to $E^1_{\pz 1} \cc J^1_1 = \id_{\ua  A'}$.
Moreover, $\lambda$ induces a chain homotopy equivalence $\Hom_R\bigl(C^m(A), \ua (A' \tensor \J)\bigr) \simeq \Hom_R\bigl(C^m(A), \ua  A'\bigr)$. Indeed, $E^1_{\pz \ast} \cc J_\ast = \id_{\Hom}$ and
\[
J^1_{\ast} \cc E^1_{\pz \ast} - \id_{\Hom} = \ba{\pa} \cc \lambda_\ast + \lambda_\ast \cc \ba{\pa}.
\]
The above equation above implies that $-\ba{\pa} J_\ast(\Tha^1_m) - c_m(H) = \ba{\pa} K^1_m$,
where $K^1_m:= \lambda \cc c_m(H)$. By construction, $\ba{\pa} K^1_m$ is a cycle in the dg submodule $(\ker E^1_{\pz \ast}, \ba{\pa})$. Since $E^1_{\pz \ast}$ is a deformation retraction, $\ker E^1_{\pz \ast}$ is acyclic, and so there exists $\wti{K}^1_m \in \ker E^1_{\pz \ast}$ such that $\ba{\pa} \wti{K}^1_m = \ba{\pa} K^1_m$.

Finally, let $\wti{H}^1_m:= J^1_1 \cc \Tha^1_m + \wti{K}^1_m$, and denote by $\wti{H} \maps C(A) \to C(A' \tensor \J)$ the coalgebra morphism corresponding to the collection of linear maps $\{H^1_1,\cdots, H^1_{m-1},\wti{H}^1_{m} \}$. Then, by construction, $(E_\po \cc \wti{H})^1_1 = (E_\po \cc H)^1_1=\ua  \cc \psi \cc \da$. Furthermore, $(E_\pz \cc \wti{H})^1_k = \Tha^1_k$  for $k=1,\ldots, m$, and $c_m({H}) = - \ba{\pa} \wti{H}^1_m$. By Prop.\ \ref{prop:obs}, the latter equation implies that 
$(\dlt'_\tensor \cc \wti{H} - \wti{H} \cc \dlt) \vert_{C^{\leq m}(A)} =0$. This completes the induction step, and hence the proof.
\end{proof}
\section{Classifying transfers of $\cP_\infty$-algebras for a  
quadratic Koszul operad $\cP$} \label{sec:Pinf} 
We describe how to generalize Thm.\ \ref{thm:isotopy}, the
classification of transfers up to isotopy, to $\Pinf$-algebras, where
$\cP$ is a symmetric operad in graded vector spaces over a field~$\kk$
with $\chark \kk=0$. Furthermore, we assume $\cP$ is a quadratic
Koszul operad \hbox{\cite[Sec.~7.2.3]{LV}}. Examples of such
$\cP_\infty$-algebras include $L_\infty$-algebras and
$C_\infty$-algebras (i.e., homotopy Lie and homotopy commutative
algebras, respectively).

Let $(A,d,\bfmu_{\cP})$ be a $\cP_\infty$-algebra and $f_1 \maps (A,d) \to (A',d')$ a quasi-isomorphism of chain complexes. Let $h$, $g_1$, and $l$ be homotopy data as in \eqref{eq:hmtpy-data}. 
A version of the transfer theorem via the ``homotopy setup'' from Sec.\ 
\ref{sec:hmtpy-setup} exists in this context provided that $g_1f_1 =\id_{A}$, and
that $h$ satisfies the {\it side conditions}: $h^2=0, \ \hbox{$f_1h = 0$},\ hg_1 = 0$  \cite[Theorem~5]{DSV:2016}. Assuming that this is the case, we obtain formulas for a transferred structure $(A',d',\bfnu_{\cP})$, and weak $\cP_\infty$-morphisms 
\[
\bff \maps (A,d,\bfmu_\cP) \rightleftarrows (A',d',\bfnu_{\cP}) \maps \bfg,
\] 
as in \eqref{eq:f} and \eqref{eq:g}. In particular, $\bff$ and $\bfg$ are lifts of $f_1$ and $g_1$, respectively, and $f_1 \cc g_1 \simeq \id_{A}$. 
We now suppose that $(A',d',\bfmu'_{\cP})$ is another $\cP_\infty$-algebra on $(A',d')$. Our goal is to determine, as in Thm.\ \ref{thm:isotopy}, whether or not it is isotopic to the transferred structure $(A',d',\bfnu_{\cP})$. 

To address this, we proceed exactly as in the proof of Thm.\ \ref{thm:isotopy}. All that we need is a suitable version of the lifting result from Prop.\ \ref{prop:lift}, and the associated 
obstruction theory behind it, which we now provide. First, in analogy with Sec.~\ref{sec:bar}, weak $\Pinf$-morphisms $(A,d,\bfmu_{\cP}) \to (A',d',\bfmu'_{\cP})$ are equivalent to dg  $\coP$-coalgebra morphisms $(\cC(A),\dlt) \to (\cC(A'),\dlt')$ \cite[Sec.\ 10.2.2]{LV}. Here $\cC(V)$ denotes the ``cofree'' $\coP$-coalgebra generated by the graded vector space $\ua  V$, where $\coP$ is the Koszul dual cooperad of $\cP$ \cite[Sec.\ 10.1.8]{LV}. Conveniently, an exact replica of Prop.\ \ref{prop:obs} for \hbox{$\cP_\infty$-algebras} is given in \cite[Thm.\ A.1]{val}. 
Next, we recall \cite[Sec.\ 3.1]{val} that the tensor product of any $\cP_\infty$-algebra with a dg commutative algebra, is also a $\cP_\infty$-algebra (cf.\ Sec.\ \ref{sec:tensor}). In particular, if $\Om(I)$ is the commutative model of the interval from Sec.\ \ref{sec:interval}, then $\bigl(A'\tensor \Om(I), d'_{\tensor},\bfmu'_{\cP \, \tensor} \bigr)$ is a $\cP_\infty$-algebra in the obvious way. 

Finally, we observe that by setting the dg algebra $\J$ to $\Om(I)$ in
the proof of Prop.\ \ref{prop:lift}, we obtain a proof of the
analogous statement for lifting chain maps to
$\cP_\infty$-morphisms. All of the required pieces are now in place to
extend the proof of Thm.\ \ref{thm:isotopy} to the $\cP_\infty$-case:

\begin{theorem} 
\label{thm:iso-Pinf} 
Given  a surjective chain homotopy
equivalence $f_1 \maps (A,d) \to (A',d')$ with homotopy data
satisfying the aforementioned ``side conditions,'' there
exists a~weak $\cP_\infty$-morphism
$\bfF \maps (A,d,\bfmu_{\cP}) \to (A',d',\bfmu'_{\cP})$ extending
$f_1$ if and only if 
the $\cP_\infty$-algebra $(A',d',\bfmu'_{\cP})$ is isotopic to a
transfer of $(A,d,\bfmu_{\cP})$ over the chain homotopy equivalence $f_1$.
\end{theorem}

\end{document}